\documentclass[11pt]{article}
\usepackage{graphicx,verbatim}
\usepackage{lineno}
\usepackage{lipsum,pdflscape}
\usepackage{amsmath,amssymb}
\usepackage{xcolor,colortbl}
\usepackage{xspace}
\usepackage{color}
\usepackage{epsfig}
\usepackage[algosection,lined,ruled,linesnumbered]{algorithm2e}
\usepackage{subfigure}

\graphicspath{{./Fig/}}
\newtheorem{theorem}{Theorem}[section]

\newtheorem{remark}[theorem]{Remark}

\newcommand{\qed}{\hfill $\square$\medskip}

\begin{document}	
\title{A New Algorithm for Computing the Frobenius Number}
	
	\author{Abbas Taheri$^1$ \and	Saeid Alikhani$^{2,}$\footnote{Corresponding author}
			}
	
	\date{\today}
	
	\maketitle
	
	\begin{center}
		$^1$Department of Electrical Engineering, Yazd University, 89195-741, Yazd, Iran \\
		$^2$Department of Mathematical Sciences, Yazd University, 89195-741, Yazd, Iran\\
	{\tt a.taheri@stu.yazd.ac.ir ~~~alikhani@yazd.ac.ir}
	\end{center}
	

\medskip
\noindent{\bf AMS Subj.\ Class.} 01B39, 11D04:

\begin{abstract}
A number $\alpha$ has a representation with respect
to the numbers $\alpha_1,...,\alpha_n$, if there exist the non-negative
integers $\lambda_1,... ,\lambda_n$ such  that
$\alpha=\lambda_1\alpha_1+...+\lambda_n \alpha_n$.
The largest natural number that does not have a representation
with respect to  the numbers $\alpha_1,...,\alpha_n$ is called the
Frobenius number and is denoted by the symbol
$g(\alpha_1,...,\alpha_n)$.
In this paper, we present a new algorithm to calculate the
Frobenius number. Also we present the sequential form of the new algorithm. 
\end{abstract}

%==================================================
% The number of pages must be 3 or 4
\section{Introduction}
Let $\alpha_1,...,\alpha_n$ $(n\geq 2)$ be positive integers
with $\gcd (\alpha_1,...,\alpha_n)=1$. Finding the largest
positive integer $N$ such that the Diophantine equation
$\alpha_1x_1 + \alpha_2x_2 +...+\alpha_nx_n=N$ has no solution
in non-negative integers is known as the Frobenius problem. Such
the largest positive integer $N$ is called the Frobenius number
of $\alpha_1,...,\alpha_n$. Various results of the Frobenius
number have been studied extensively.

The Frobenius problem is well known as the coin problem that
asks for the largest monetary amount that cannot be obtained
using only coins in the set of coin denominations which has no
common divisor greater than $1$. This problem is also referred
to as the McNugget number problem introduced by Henri Picciotto.
The origin of this problem for $n=2$ was proposed by Sylvester
(1884), and this was solved by Curran Sharp (1884), see
\cite{Sharp,Syl}.
Curran Sharp  \cite{Sharp} in 1884  proved that
$ g(\alpha_1,\alpha_2)=\alpha_1\alpha_2-\alpha_1-\alpha_2$.

Let $\alpha_1,...,\alpha_n$ be positive integers whose greatest
divisor is equal to one, in other words
\[
\gcd(\alpha_1,...,\alpha_n)=1.
\]
If $S=<\alpha_1,...,\alpha_n>$ 
is the semigroup generated by $\alpha_1,...,\alpha_n$, then
finding $g(S)$ is a problem and therefore finding the bounds for
$g( S)$, whenever we have a certain sequence of numbers,
\cite{Eur} is of interest. For example, if $S$ is an arithmetic
sequence with relative value $d$, then we have \cite{Robert}:
\[
g(a,a+d,a+2d,...,a+kd)=a\lfloor\frac{a-2}{k} \rfloor+d(a-1).
\]

Fibonacci sequence is a recursive sequence
$F_n=F_{n-1}+F_{n-2}$, $n\geq 3$
with
$F_1=F_2=1$. 
For every integer $l\geq i+2$, we have
$g(F_i,F_{i+1},F_l)=g(F_i,F_{i+ 1})$.
Assuming $gcd(F_i,F_j,F_l)=1$ for the triplet $3\leq i<j<l$,
calculating
 $g(F_i ,F_j,F_l)$ has been considered.
	Suppose that $i,k\geq 3$ are integers and
	$r=\lfloor\frac{F_i-1}{F_i}\rfloor$.
	In this case (see \cite{Amita})
	\[
	g(F_i,F_{i+2},F_{i+k})=\left\{
	\begin{array}{lr}
	{\displaystyle (F_i-1)F_{i+2}-F_i(rF_{k-2}+1)};&
\quad\mbox{If $r=0$ or $r\geq 1$ and} \\ {} &
\quad\mbox{$F_{k-2}F_i<(F_i -rF_k)F_{i+2},$}\\[15pt]
	{\displaystyle r(F_{k}-1)F_{i+2}-F_i((r-1)F_{k-2}+1)}&
	\quad\mbox{otherwise}
	\end{array}
	\right.
	\]

It was shown by Curtis \cite{5} that no closed formula exists for the
Frobenius number if $n>2$. Because of this reason,  there has been a great deal of research into producing upper
bounds on $g(a_1,a_2,...,a_n)$. These bounds share the property that in the worst-case they are of quadratic order with
respect to the maximum absolute valued entry of $(a_1,...,a_n)$. Assuming that $a_1\leq a_2\leq ...\leq a_n$ holds, such bounds include the
classical bound by Erd\H{o}s and Graham \cite{6}
\[
g(a_1,...,a_n)\leq  2a_{n-1}\lfloor\frac{a_n}{n}\rfloor - a_n,
\]
by Selmer \cite{12}
\[
g(a_1,...,a_n)\leq  2a_{n}\lfloor\frac{a_1}{n}\rfloor - a_1,
\]
by Vitek \cite{16}
\[
g(a_1,...,a_n)\leq \frac{1}{2} (a_2-1)(a_n-2)-1,
\]
and by Beck et al. \cite{2}
\[
g(a_1,...,a_n)\leq \frac{1}{2} \big(\sqrt{a_1a_2a_3(a_1+a_2+a_3)}-a_1-a_2-a_3\big).
\]

In Section 2, we present a new algorithm to compute the Frobenius number. Also we  present the sequential form of the new algorithm  in Section 3.

\section{New algorithm}
In this section, we present a new algorithm to calculate the
Frobenius number. 
We start this section with the following easy theorem: 
\begin{theorem}\label{new}
For the numbers    $\alpha_1<\alpha_2<...<\alpha_n$, we have 
	\[
	g(\alpha_1,...,\alpha_n)\leq g(\alpha_1,...,\alpha_{n-1})\leq...\leq g(\alpha_1,\alpha_2).
	\]
\end{theorem}

\begin{proof}
	By the definition of Frobenius number,  all integers strictly greater than 
	$g(\alpha_1,\alpha_{2})=\alpha_1\alpha_2-\alpha_1-\alpha_2$ 
	can be expressed as $\alpha_ix_i + \alpha_j x_j$ for some $x_i, x_j \in \mathbb{Z}^+$. So it follows that all integers strictly
	greater than $\alpha_1\alpha_2-\alpha_1-\alpha_2$ can be expressed as $\sum_{k=1}^n a_kx^k$ for nonnegative integer $x_k$ where $ k \in \{1, 2,...,n\}$. Therefore, we have the result. \qed
		\end{proof}

Using the upper bound of Theorem \ref{new}, we present a new
algorithm for calculating Frobenius numbers.
More precisely, since $g(\alpha_1,...,\alpha_n)\leq
g(\alpha_1,\alpha_{2})=\alpha_1\alpha_2-\alpha_1-\alpha_2,$ we
compute the number $\alpha_1\alpha_2-\alpha_1-\alpha_2$
and using our sub-algorithm which we call it {\it HasRep Algorithm} examine the natural
numbers less than $g(\alpha_1,\alpha_{2})$ are representable
respect to $\alpha_1,...,\alpha_n$ or not. Obviously, the largest number
less than $\alpha_1\alpha_2-\alpha_1-\alpha_2$ which dose not
have a representation, is the Frobenius number of
$\alpha_1,...,\alpha_n$. 
After this, by the Algorithm 2.\ref{2} which we call it {\it  Frob Algorithm}, we
apply the Algorithm  2.\ref{1} to compute the Frobenius number
of $\alpha_1,...,\alpha_n$.

\begin{algorithm}[htb]
		\scriptsize
	\label{2}	
	%\SetAlgoLined	
	\SetKwInOut{Input}{input}
	\Input{ L, the list of numbers.}
	\SetKwInOut{Output}{output} 
	\Output{The Frobenius number of L}
%\hspace*{\algorithmicindent} \textbf{Input:}  L, the list of numbers\;
%\hspace*{\algorithmicindent} \textbf{Output:} The Frobenius number of L\;
% \KwData{this text}
% \KwResult{how to write algorithm with \LaTeX2e }

$F:=L[1]*L[2]-(L[1]+L[2])$;

Flag:=true;

\textbf{for} $a$ \textbf{from} $F$ \textbf{by} $-1$ \textbf{to} $L[1]+1$ \textbf{while} flag \textbf{do}

  $Flag:=HasRep(a,L)$;

\textbf{end for}

return(a+1);

End.

 \caption{Frob Algorithm}
\end{algorithm}

We have implemented the new algorithm in the Maple software.
We briefly show the results of the implementation of this
algorithm for a few examples in the Table \ref{table1}.  It should be note that the designed
algorithm, unlike some algorithms, is responsible for any number
of numbers.

\begin{algorithm}[H]
		\scriptsize
		\label{1}	
		%\SetAlgoLined	
		\SetKwInOut{Input}{input}
		\Input{$a,L$: $a$ number and $L$ list of numbers.}
		\SetKwInOut{Output}{output} 
		\Output{ture, if $a$ has a representation with respect to $L$ and false, otherwise.}

% \KwData{this text}
% \KwResult{how to write algorithm with \LaTeX2e }

\textbf{if} $\# L = 2$ \textbf{then}
	 
	 flag:= false;
	 
	 \textbf{if} $a~ mod ~L[2]=0$ \textbf{then} 
	
	 return(true) ; 
	
	 \textbf{elif} $a=L[1]*L[2]-L[1]-L[2]$ \textbf{then}
	 
	 return(false);
	 
 \textbf{fi;} 
	 
	 \textbf{while} $a\geq L[2]$ \textbf{and not flag do}
	 
	 \textbf{if} $a~ mode~ L [1] \neq 0$ \textbf{then}
	 
	 $a: = a –L[2]$;
	 
	 \textbf{else} 
	 
	 return(true); 
	 
	 \textbf{fi;} 
	 
	\textbf{end while;}
	 
	 return(false); 
	 
	 \textbf{else}
	
	 \textbf{if} $a ~ mode ~L[-1]= 0$ or $HasRep(a, L[1..-2])$  \textbf{then}
	
	 return(true); 
	
	 \textbf{else}
	
	 flag:= false;    
	
	 \textbf{while} $a \geq L[-1]$ \textbf{and not flag do} 
	 
	 $a:=a-L[-1];$  
	
 	 $flag:=HasRep(a,L[1..-2])$;        
	
	 \textbf{end while;}
	
  \textbf{fi;}
	
 	 return (flag) ;
	
	 \textbf{fi\;}
End.
 \caption{HasRep}
\end{algorithm}

\begin{table}[htb]
\begin{center}
    \begin{tabular}[t]{|p{10.2cm}|c|c|c|}
\hline
    Numbers & Frobenius Number \\ \hline
    7,11,13 & 30    \\ \hline
    53,71,91 & 899\\ \hline
	322, 654, 765 & 27971  \\ \hline
 	123,1234,12345	&71459	\\ \hline 
	151, 157, 251, 711& 	3019 \\ \hline
	151, 157, 251, 711, 912 &	3019 \\ \hline
101,109,113,119,121,131,139,149,151,161,163,167,169,187,191, 214,219,238,276,324,345,346,349,387,421,427,444,453,463,525, 530,555,579,580,625,711,719,737,752,787,814,834,856,878,899, 915,937,978,989 &426 \\

    \hline      
    \end{tabular} 
    \caption{The results of the implementation of proposed 
    	algorithm for a few examples }\label{table1}
  \end{center}
\end{table}

\section{The sequential form of algorithm }

It is interesting that the proposed algorithm can convert to  the sequential form, that we do it in this section.
Since the HasRep Algorithm gives YES or NO, we use a function to give us $0$ and  $1$. We define the function $f$ as follows:

\[
f(\alpha_1,R) := \lfloor\frac{\alpha_1}{R}  -  \frac{1}{\lfloor\frac{R}{\alpha_1} \rfloor} \rfloor.
 \]
Note that if $R$ is divisible by $\alpha_1$, then the result of the function $f(\alpha_1,R) $ is zero. Otherwise, the result of the expression is between $-1$ and $0$. 
Now we define the function $H$ as follows:

\[
H(R,[\alpha_1, \alpha_2]) :=  \lfloor\frac{\alpha_1}{R}  -  \frac{1}{\lfloor\frac{R}{\alpha_1} \rfloor} \rfloor   \lfloor\frac{\alpha_2}{R}  -  \frac{2}{\lfloor\frac{R}{\alpha_2} \rfloor} \rfloor  \lfloor\frac{\alpha_1}{R-\alpha_2}  -  \frac{1}{\lfloor\frac{R-\alpha_2}{\alpha_1} \rfloor} \rfloor \lfloor\frac{\alpha_1}{R-2\alpha_2}  -  \frac{1}{\lfloor\frac{R-2\alpha_2}{\alpha_1} \rfloor} \rfloor ... \]
\[  \lfloor\frac{\alpha_1}{R-(\lfloor\frac{R}{\alpha_2} \rfloor -1)  \alpha_2}  -  \frac{1}{\lfloor\frac{R-(\lfloor\frac{R}{\alpha_2} \rfloor -1)  \alpha_2}{\alpha_1} \rfloor} \rfloor  {\lfloor\frac{R-\lfloor\frac{R}{\alpha_2} \rfloor  \alpha_2}{\alpha_1} \rfloor \alpha_1} - {R+\lfloor\frac{R}{\alpha_2} \rfloor  \alpha_2}
\]

The function $H$ has two variables  (two input) which investigate the linear representation of $R$ with respect to a list. It is obvious that if $R$ has a linear representation 
with respect to $L$, then the value of the function $H$ is $0$, otherwise is a number in $(-1,0)$ or in $(0,1)$. Now we define another function which we denote it by $N$ (it inverts the answer of the $H$ function) as follows:  

\begin{equation*}
	N(x)=
	\begin{cases}
		1 & \text{if } x=0\\
		0 & \text{if } \{x | -1 \leq x \leq 1 ; x\neq0\}
	\end{cases}
\end{equation*}
Note that 
\[ 
 N(x) := \lfloor{-|x|} \rfloor +1.
  \]

Based the Frob Algorithm, we first find the upper bound $U$ and then enter the numbers from $1$ to $U$ into the function $N(H(U,[\alpha_1,\alpha_2]))$. 
The numbers that remain have no representation with respect to the list. But the Frobenius number is the largest among them, so we can calculate it. 
If 
$ \delta_i := N(H(i,[\alpha_1,\alpha_2]))$,  
then the following theorem gives the Frobenius number $g(\alpha_1,...,\alpha_n)$: 
%\[ \lambda_i := N(H(i,[\alpha_1,\alpha_2])) ;\] 

\begin{theorem}
	The Frobenius number $g(\alpha_1,...,\alpha_n)$ based on the values of $\delta_i$ and function $N(\delta_i)$ is equal to
\[ g(\alpha_1,...,\alpha_n):= (U)(\delta_U) + (U-1)\times(\delta_{U-1})N(\delta_U) +\cdots+ (1)(\delta_1)N(\delta_U)...N(\delta_2)  \]
\end{theorem} 

\medskip

We close the paper by the following remark:  

\begin{remark} 
	The   function $H$ can be generalized for  $n$ numbers, i.e., if  
$ \delta_i := N(H(i,[\alpha_1,...,\alpha_n])),$, then  
\[
  H(R,[\alpha_1, \alpha_2])\times H(R,[\alpha_1, \alpha_2,\alpha_3])\times...\times H(R,[\alpha_1,..., \alpha_{n-1}])\times\lfloor { \frac{\alpha_n}{R} - \frac{1}{ \lfloor { \frac{R}{\alpha_n} } \rfloor } }  \rfloor \times 
 \] 
\[  \prod_{i=1}^{\lfloor { \frac{R}{\alpha_n} } \rfloor}H(R-i\alpha_n,[\alpha_1,...,\alpha_{n-1}])  \]

\end{remark}

%==================================================

%==================================================


\begin{thebibliography}{5}
	
	
	
\bibitem{2}	M. Beck, R. Diaz,  S. Robins, The Frobenius problem, rational polytopes, and fourier 
	dedekind sums,  Journal of number theory, 96(1) (2002) 1–21.
	
	
	\bibitem{5} F. Curtis, On formulas for the Frobenius number of a numerical semigroup,  Mathematica Scandinavica,
	67(2) ( 1990)190–192.
	
\bibitem{6}  P. Erd\H{o}s,  R. Graham, On a linear diophantine problem of frobenius, Acta Arithmetica,
(1972)	21(1) 399–408.
	
	
	

\bibitem{Eur}L. Fukshansky , A. Sch\" urmann, Bounds on
generalized Frobenius numbers, Eur. J. Combin, 32 (3) (2011)
361--368.


\bibitem{Robert} J.B. Roberts, Note on linear forms, Proc. Amer.
Math. Soc. 7 (1956),
465-469. 

\bibitem{Sharp} W.J. Curran Sharp, Solution to problem 7382
(Mathematics). Educational Time, (1884) 41.

\bibitem{12} E.S. Selmer,  On the linear diophantine problem of frobenius. Journal f\"{u}r die reine und angewandte
Mathematik,  (293-294) (1977) 1–17.

\bibitem{Syl} J.J. Sylvester, Problem 7382, Mathematical
Questions
from the Educational Times, (1884) 41, 21.

\bibitem{Amita} A. Tripathi, Formulae for the Frobenius number
in three variables, J. Number Theory 170 (2017) 368-389.

\bibitem{16} Y. Vitek, Bounds for a linear diophantine problem of frobenius, Journal of the London Mathematical Society, 2(1) (1975) 79-85.
	
\end{thebibliography}
\end{document}